\documentclass[11pt]{article}
\usepackage{amsmath}
\usepackage{amsthm}
\usepackage{amsfonts}
\usepackage[latin5]{inputenc}
\usepackage{color}
 \usepackage{amssymb}

\usepackage{fancyhdr}

\begin{document}
\numberwithin{equation}{section}
\newcounter{thmcounter}
\newcounter{Remarkcounter}
\newcounter{Defcounter}
\numberwithin{thmcounter}{section}
\newtheorem{Prop}[thmcounter]{Proposition}
\newtheorem{Corol}[thmcounter]{Corollary}
\newtheorem{theorem}[thmcounter]{Theorem}
\newtheorem{Lemma}[thmcounter]{Lemma}
\theoremstyle{definition}
\newtheorem{Def}[Defcounter]{Definition}
\theoremstyle{remark}
\newtheorem{Remark}[Remarkcounter]{Remark}
\newtheorem{Example}[thmcounter]{Example}

\newcommand{\COD}[3]{Codazzi equation \eqref{MinkCodazzi} for $X=e_{#1}$, $Y=e_{#2}$ and $Z=e_{#3}$}
\newcommand{\diag}{\mathrm{diag}\ }
\newcommand{\DIV}{\mathrm{div}}
\title{Generalized constant ratio hypersurfaces in Euclidean spaces}
\author{Nurettin Cenk Turgay}
\date{}

\maketitle
\begin{abstract}
In this paper, we study generalized constant ratio (GCR) hypersurfaces in Euclidean spaces. We mainly focus on the hypersurfaces in $\mathbb E^4$. First, we deal with  $\delta(2)$-ideal GCR hypersurfaces. Then, we  study on hypersurfaces with constant (first) mean curvature. Finally, we obtain the complete classification of  GCR hypersurfaces with vanishing Gauss-Kronecker curvature. We also give some explicit examples.


\end{abstract}

\section{Introduction}\label{SectionIntrod}
One of the most basic objects studied to understand geometrical properties of a (semi-)Riemannian submanifold of a (semi-)Euclidean space is its position vector. In this direction, the notion of constant ratio (CR) submanifolds in Euclidean spaces introduced by B.-Y. Chen, \cite{ChenCRSurf2001}.  Let $M$ be a submanifold of the Euclidean space $\mathbb E^{m}$ and $x:M\rightarrow \mathbb E^{m}$ its position vector.  Put $x=x^T+x^\perp$, where $x^T$ and $x^\perp$ denote the tangential and normal components of $x$, respectively. If the ratio of length of these two vectors is constant, then $M$ is said to be a CR submanifold  or an equiangular submanifold (\cite{Boyadzhieb2007}). Some results in CR submanifolds appeared also in \cite{ChenCRSurf2002,ChenCRSurf}. 

On the other hand,  in the particular case of the codimension 1,  if $M$ is an CR hypersurface, then the tangential part $x^T$ of $x$ is a principal direction of $M$. However, the converse of this statement does not hold in general. For example, it is proved that all of rotational surfaces in $\mathbb E^3$ have this property  (see \cite[Proposition 4.2]{YuFu2014GCRS}), but a rotational surface is not a CR surface unless its profile curve is chosen specifically. Therefore, the definition of  generalized constant ratio (GCR) surface  has been recently given: If $x^T$ is a principal direction of a surface, then the surface is said to be a GCR  surface, \cite{YuFu2014GCRS}.

We would like to note that if the ambient space is Euclidean, being generalized constant ratio of a surface is equivalent to having canonical principal direction (CPD) (see \cite[Proposition 2.1]{YuFu2014GCRS}). Surfaces with CPD were studied in some articles appeared in more recent times. For example, in \cite{DFJvK2009,DMuntNistr2009}, several authors obtained the classification of surfaces with canonical principal direction in the spaces $\mathbb S^2 \times\mathbb  R$ and $\mathbb H^2 \times\mathbb  R$, respectively. Furthermore,  generalized constant ratio  surfaces in the Euclidean 3-space were studied in \cite{YuFu2014GCRS,MunteanuNistor2011}.

In this paper, we consider  generalized constant ratio hypersurfaces. First, we obtain some general results in the Euclidean space of arbitrary dimension. We also show that GCR hypersurfaces and biconservative hypersurfaces satisfy an interesting common property (see Proposition \ref{PropConnBicandGCR}). Then, we give our main results and explicit examples in the Euclidean 4-space. 

The organization of this paper is as follows. In Sect. 2, we give basic definitions and notation. We also obtain a lemma which we use while obtaining further results. In Sect. 3, after we describe the notion of generalized constant ratio  hypersurfaces, we prove an equivalent condition to being GCR  which is a generalization of \cite[Proposition 2.1]{YuFu2014GCRS} to higher dimensions.  In Sect. 4, first, we obtain some explicit examples. Then, we present our main results in the Euclidean space $\mathbb E^4$. 

The submanifolds we are dealing with are smooth and connected unless otherwise stated.

\section{Prelimineries}
Let $\mathbb E^m$ denote the Euclidean $m$-space with the canonical Euclidean metric tensor given by  
$$
\widetilde g=\langle\ ,\ \rangle=\sum\limits_{i=1}^m dx_i^2,
$$
where $(x_1, x_2, \hdots, x_m)$  is a rectangular coordinate system in $\mathbb E^m$. 

Consider an $n$-dimensional Riemannian submanifold $M$  of the space $\mathbb E^m$. We denote Levi-Civita connections of $\mathbb E^m$ and $M$ by $\widetilde{\nabla}$ and $\nabla$, respectively. Then, Gauss and Weingarten formulas are given, respectively, by
\begin{eqnarray}
\label{MEtomGauss} \widetilde\nabla_X Y&=& \nabla_X Y + h(X,Y),\\
\label{MEtomWeingarten} \widetilde\nabla_X \xi&=& -S_\xi(X)+\nabla^\perp_X \xi
\end{eqnarray}
for all  $X,\ Y\in \Gamma(TM)$ and  $\xi\in \Gamma(T^\perp M)$,  where  $\Gamma(TM)$ and $\Gamma(T^\perp M)$ denote the space of vector fields tangent to $M$ and normal to $M$, respectively and $h$,  $\nabla^\perp$ and $S$ are the second fundamental form, the normal connection and  the shape operator of $M$,  respectively. Note that for each $\rho \in T^{\bot}_m M$, the shape operator $S_{\rho}$ along the normal direction $\rho$  is a symmetric  endomorphism of the tangent space  $T_m M$ at $m \in M$ which is related with the second fundamental form by 
\begin{eqnarray}
\label{AhRelated}\left\langle h(X_m, Y_m), \rho\right\rangle = \left\langle S_{\rho}X_m, Y_m \right\rangle.
\end{eqnarray}
 
On the other hand, since the curvature tensor of the ambient space $\mathbb E^m$ vanishes identically, Gauss and Codazzi equations become
\begin{eqnarray}
\label{MinkGaussEquation}\label{GaussEq} \langle R(X,Y)Z,W\rangle&=&\langle h(Y,Z),h(X,W)\rangle-
\langle h(X,Z),h(Y,W)\rangle,\\
\label{MinkCodazzi} (\bar \nabla_X h )(Y,Z)&=&(\bar \nabla_Y h )(X,Z),
\end{eqnarray}
respectively, where  $R$ is the curvature tensor of $M$ and  $\bar \nabla h$ is defined by
$$(\bar \nabla_X h)(Y,Z)=\nabla^\perp_X h(Y,Z)-h(\nabla_X Y,Z)-h(Y,\nabla_X Z).$$

We would like to prove the following lemma which we will use. In this lemma, we put
$$\begin{array}{rcl}\nabla X:\Gamma(TM)&\rightarrow& \Gamma(TM)\\Y&\mapsto&\nabla_YX\end{array}$$
 and the index $i$ (resp. $\alpha$)  runs over the range $1,2,\hdots,p$ (resp. $1,2,\hdots,q$).
\begin{Lemma} \label{WarpPrdctKeyLemma}
Let $M$ be a $p+q+1$ dimensional Riemannian manifold. Assume that there exist a unit vector field $X$ and two involutive distributions $T_1$ and $T_2$ of dimension $p$ and $q$ such that 
\begin{enumerate}
\item [(i).]   $T_mM=X|_m\oplus T_1(m)\oplus T_2(m)$ for all $m\in M$,
\item [(ii).]  $\nabla_{T_{3-i}}T_i\subset T_i$ for $i=1,2$,
\item [(iii).] The linear mapping  $\nabla X$ satisfies $\nabla X|_{T_i}=\lambda I$ and $\nabla X|_{T_i}=\nu I$ for some smooth functions $\lambda$ and $\nu$ satisfying
 $$Y(\lambda)=Z(\nu)=0 \mbox{ \quad whenever $Y\in T_1$ and $Z\in T_2$},$$
 where $I$ is the identity operator acting on $\Gamma(TM)$,
\item [(iv).] $X=\partial_s$ for  a smooth function $s:M\rightarrow \mathbb R$.
\end{enumerate}
Then, there exists a local coordinate system $(s,t_1,t_2,\hdots,t_p,u_1,u_2,\hdots,u_q)$ such that
\begin{enumerate}
\item [(A).] $X=\partial_s$, $T_1=\mathrm{span}\{\partial_{t_1},\partial_{t_2},\hdots,\partial_{t_p}\}$ and $T_2=\mathrm{span}\{\partial_{u_1},\partial_{u_2},\hdots,\partial_{u_q}\}$,
\item [(B).] $\nabla_{\partial_{u_\alpha}}\partial_{t_i}=\nabla_{\partial_{t_i}}\partial_{u_\alpha}=0$,
\item [(C).] The metric tensor of $M$ is
\begin{equation}\label{WarpPrdctKeyLemmaResult}
g=ds^2+a(s)\left(\sum\limits_i A_i\left(t_1,t_2,\hdots,t_p\right)dt_i^2\right)+ b(s)\left(\sum\limits_\alpha B_\alpha\left(u_1,u_2,\hdots,u_q\right)du_i^2\right)
\end{equation}
for some non-vanishing smooth functions $a,A_i,b,B_\alpha$.
\end{enumerate}
\end{Lemma} 
\begin{proof}
Because of conditions (i), (ii) and (iv),  by employing \cite[Lemma 4.2]{Magid1985}, one can see that there exists a local coordinate system $(s,t_1,t_2,\hdots,t_p,u_1,u_2,\hdots,u_q)$ such that (A) is satisfied. Thus, the metric tensor of $M$ becomes
\begin{equation}\label{WarpPrdctKeyLemmaBe4Result}
g=ds^2+\sum\limits_i\hat a_i(s,t,u)dt_i^2+\sum\limits_\alpha \hat b_\alpha(s,t,u)du_\alpha^2,
\end{equation}
where we put $t=(t_1,t_2,\hdots,t_p)$ and $u=(u_1,u_2,\hdots,u_q)$. 
On the other hand, since $[{\partial_{u_\alpha}},\partial_{t_i}]=0,$ we have $\nabla_{\partial_{u_\alpha}}\partial_{t_i}=\nabla_{\partial_{t_i}}\partial_{u_\alpha}$. However, because of the condition (ii), we have $\nabla_{\partial_{u_\alpha}}\partial_{t_i}\in T_1$ and $\nabla_{\partial_{t_i}}\partial_{u_\alpha}\in T_2$. Hence, we have (B).

Moreover, (B) implies  $
2\hat a_i\partial_{u_\alpha}\hat a_i=\partial_{u_\alpha}(\langle \partial_{t_i},\partial_{t_i}\rangle)
=2\langle \nabla_{\partial_{u_\alpha}}\partial_{t_i},\partial_{t_i}\rangle=0.$ 
Therefore, we have 
\begin{equation}\label{WarpPrdctKeyLemmaBe4Result00}
\hat a_i=\hat a_i(s,t).
\end{equation}
 On the other hand, because of the condition (iii), we have $\nabla_{\partial_{t_i}}\partial_s=\lambda{\partial_{t_i}}$. This equation and \eqref{WarpPrdctKeyLemmaBe4Result} imply
$$\frac 12\partial_s(\hat a_i^2)=\frac 12\partial_s\left(\langle\partial_{t_i},{\partial_{t_i}}\rangle\right)=\langle\nabla_{\partial_s}\partial_{t_i},{\partial_{t_i}}\rangle=\langle\nabla_{\partial_{t_i}}\partial_s,{\partial_{t_i}}\rangle=\langle\lambda\partial_{t_i},{\partial_{t_i}}\rangle=\lambda \hat a_i^2$$
from which we get
\begin{equation}\label{WarpPrdctKeyLemmaBe4Result01}
\lambda=\partial_s(\ln \hat a_i).
\end{equation}
 Since $\partial_{t_i}(\lambda)=0$, \eqref{WarpPrdctKeyLemmaBe4Result00} and \eqref{WarpPrdctKeyLemmaBe4Result01} implies $\hat a_i=a(s) A_i(t)$ for some smooth functions $A_i,a$. Analogously, we have, $\hat b_\alpha=b(s) B_\alpha(t)$ for some smooth functions $B_\alpha,b$. Hence, \eqref{WarpPrdctKeyLemmaBe4Result} implies (C).
\end{proof}


\subsection{Hypersurfaces in Euclidean spaces}
Now, let $M$ be an oriented hypersurface in the Euclidean space $\mathbb E^{n+1}$ and  $S$ its shape operator along the unit normal vector field $N$ associated with the oriantiation of $M$. We consider a local orthonormal frame field $\{e_1,e_2,\hdots,e_n;N\}$ consisting of  principal directions of $M$ with corresponding principal curvatures $k_1,\ k_2,\hdots,\ k_n$.  Then, the shape operator $S$ of $M$ becomes
$$S=\diag(k_1,k_2,\hdots,k_n)$$ 
and the $k$-th invariant of $S$ is the function defined by
$$s_k=\sum\limits_{i_1<i_2<\hdots<i_k}k_{i_1}k_{i_2}\hdots k_{i_k},\quad k=1,2,\hdots,n.$$
 By using these invariants, mean curvatures $H_1,H_2,\hdots,H_n$ of $M$ are defined by
 ${n\choose k}H_k=s_k.$
(see, for example, \cite{AliasGurbuzGeomDed2006}). Note that $H_1$, the first mean curvature of $M$, is usually called as the mean curvature of $M$, while $H_n$ is called the Gauss-Kronecker curvature of $M$ for $n>2$ and $M$ is said to be $k$-minimal if $H_k=0$. 

$M$ is said to be isoparametric if all of its principal curvatures are constant. It is well-known that if $M$ is an isoparametric hypersurface of $\mathbb E^{n+1}$, then it is congruent to an open part of $ \mathbb S^{n-p}(r^2)\times \mathbb E^p$,  where $r\neq 0$ is a constant and $p\in\{0,1,\hdots,n\}$.

On the other hand, if  $\{\theta_1,\theta_2,\hdots,\theta_n\}$ is a dual basis of the local, orthonormal, base field  above,  then the first structural equation of Cartan is given by
\begin{equation}
\label{CartansFirstStructural}d\theta_i=\sum\limits_{i=1}^n\theta_j\wedge\omega_{ij},\quad i=1,2,\hdots,n,
\end{equation}
where $\omega_{ij}$ denotes connection forms corresponding to the chosen frame field, i.e., $\omega_{ij}(e_l)=\langle\nabla_{e_l}e_i,e_j\rangle$ and satisfies $\omega_{ij}=-\omega_{ji}$.

\subsection{$\delta(r)$-ideal hypersurfaces}
Let $M$ be an $n$-dimensional submanifold of the Euclidean space $\mathbb E^{m}$ and $\delta (r)$ denote the $r$-th $\delta$ invariant introduced in \cite{ChenPRGeom2001}, where B.-Y. Chen proved the general sharp inequality
\begin{equation}\label{BYCSharpIneq}
\delta(r)\leq \frac{n^2(n-r)}{2(n-r + 1)}\langle H,H\rangle,\quad r=2,3,\hdots,n-1
\end{equation}
if $H=1/n(\mathrm{tr}h)$ denotes the mean curvature vector of $M$ in $\mathbb E^{m}$ (see also \cite{ChenMunteanuBI}). Moreover, he called a submanifold $M$ as $\delta(r)$-ideal if the equality case of \eqref{BYCSharpIneq} is satisfied identically. We would like to state the following lemma obtained in \cite{ChenPRGeom2001}.
\begin{Lemma}\cite{ChenPRGeom2001}\label{Delta2InvLemma}
Let $M$ be a hypersurface in the Euclidean space $\mathbb E^{n+1}$. Then, it is a $\delta(2)$-ideal hypersurface if and only if its principal curvatures are $k_1,k_2,\underbrace{k_1+k_2,k_1+k_2,\hdots,k_1+k_2}_{(n-2)\mbox{ times}}.$
\end{Lemma}

\section{Generalized constant ratio hypersurfaces}
In this section, we, firstly, want to extend the definition of generalized constant ratio surfaces given in \cite{YuFu2014GCRS} by Fu and Munteanu to hypersurfaces in Euclidean spaces. 

Let $M$ be an oriented, immersed hypersurface in the Euclidean space $\mathbb E^{n+1}$,  $x:M\rightarrow \mathbb E^{n+1}$ an isometric immersion and $N$ its unit normal vector. We define functions $\mu$ and $\theta$ by
\begin{eqnarray}
\label{DefmuExp1}\mu&=&\langle x,x \rangle^{1/2},\\
\label{DefthetaExp1}\mu\cos\theta&=&\langle x,N \rangle.
\end{eqnarray}
Then, the position vector $x$ can be expressed as
\begin{eqnarray}
\label{Dexompofx}x=\mu\sin\theta e_1 +\mu\cos\theta N,
\end{eqnarray}
where  $e_1$  is a unit tangent vector field along the tangential component $x^T$ of the vector field $x$.
Now, let $X$ be a vector field tangent to $M$. Since $\widetilde\nabla_Xx=X$, \eqref{Dexompofx} implies
\begin{align}\label{DexompofxRes1}
\begin{split}
 X=&X(\mu\sin\theta) e_1 +\mu\sin\theta\nabla_{X}e_1-\mu\cos\theta SX\\&
+\mu\sin\theta h(X,e_1)+X(\mu\cos\theta)N.
\end{split}
\end{align}
Note that by applying $X$ to \eqref{DefmuExp1} and using \eqref{Dexompofx}, one can get
\begin{eqnarray}
\label{DexompofxRes2}X(\mu)=\langle X,e_1\rangle \sin\theta,
\end{eqnarray}
i.e., $Y(\mu)$ vanishes identically whenever $Y$ is a vector field tangent to $M$ and orthogonal to $e_1$. Therefore, by combining \eqref{DexompofxRes1} and \eqref{DexompofxRes2}, we obtain
\begin{align}\label{DexompofxRes3}
\begin{split}
 X=&\langle X,e_1\rangle \sin^2\theta e_1+\mu\cos\theta X(\theta) e_1 +\mu\sin\theta\nabla_{X}e_1-\mu\cos\theta SX\\&
+\mu\sin\theta h(X,e_1)+\langle X,e_1\rangle \sin\theta\cos\theta N-\mu\sin\theta X(\theta)N.
\end{split}
\end{align}

Next, we obtain the following proposition.
\begin{Prop}\label{PROPOPP1}
Let $M$ be an oriented hypersurface in the Euclidean space $\mathbb E^{n+1}$ and  $x:M\rightarrow \mathbb E^{n+1}$ its position vector with the tangential component $x^T$. Then,  $x^T$ is a principal direction if and only if $Y(\theta)=0$ whenever $Y$  is a vector field tangent to $M$ and orthogonal to $x^T$.
\end{Prop}
\begin{proof}
By the notation described above, the normal part of the equation \eqref{DexompofxRes3} for $X=Y$ gives 
\begin{eqnarray}
\label{DexompofxRes4}\mu\sin\theta h(Y,e_1)+\langle Y,e_1\rangle \sin\theta\cos\theta N-\mu\sin\theta Y(\theta)N=0.
\end{eqnarray}
Note that $x^T$  is a principal direction of $M$ if and only if  $h(Y,e_1)=0$ whenever $Y$ is orthogonal to $e_1$ because of \eqref{AhRelated}. However,  from \eqref{DexompofxRes4} one can observe that this condition is equivalent to   $Y(\theta)=0$. 
\end{proof}
\begin{Remark}
See  \cite[Proposition 2.1]{YuFu2014GCRS} and \cite[Theorem 1]{MunteanuNistor2011} for the case $n=2$.
\end{Remark}
Therefore, we would like to give the following definition.
\begin{Def}
A hypersurface is said to be a  generalized constant ratio (GCR) hypersurface if the tangential part $x^T$ of its position vector is one of its principal directions, or, equivalently, by the notation above,   $Y(\theta)=0$ whenever $Y\in\Gamma(TM)$  is orthogonal to $x^T$.
\end{Def}

Now, let $M$ be a GCR hypersurface in the Euclidean space $\mathbb E^{n+1}$ and $\{e_1,e_2,\hdots,e_n;N\}$ be the local orthonormal frame field consisting of the principal directions of $M$ with corresponding principal curvatures $k_1,k_2,\hdots,k_n$. We assume that $e_1$ is proportional to the tangential component $x^T$ of the position vector $x$ of $M$. Then, from \eqref{DexompofxRes3} for $X=e_1$, one can obtain 
\begin{subequations}\label{GCRCondall12}
\begin{eqnarray}
\label{GCRCond1} \nabla_{e_1}e_1&=&0,\\
\label{GCRCond2} k_1&=&e_1(\theta)-\frac{\cos\theta}\mu.
\end{eqnarray}
\end{subequations}

We state the following proposition which is a direct result of \eqref{GCRCondall12}.
\begin{Prop}\label{PropConnBicandGCR}
Let $\alpha$ be an integral curve of $e_1$. Then, $\alpha$  is a   geodesic and a line of curvature of $M$. Furthermore, $\alpha$ is planar and its curvature in $\mathbb E^{n+1}$ is $\left.k_1\right|_\alpha$
\end{Prop}
\begin{Remark}
Because of this proposition, one can expect that several geometrical properties of  constant ratio hypersurfaces are coinciding with  geometrical properties of biconservative hypersurfaces (see \cite[Lemma 2.2]{HsanisField}).
\end{Remark}
Furthermore, \eqref{DexompofxRes2}, \eqref{DexompofxRes3} for $X=e_i$ and Proposition \ref{PROPOPP1} imply
\begin{subequations}\label{GCRCondall34}
\begin{eqnarray}
\label{GCRCond3} e_i(\theta)&=&e_i(\mu)=0,\\
\label{GCRCond4}  \nabla_{e_i}e_1&=&\frac{1+\mu\cos\theta k_i}{\mu\sin\theta} e_i,\quad i=2,3,\hdots,n.
\end{eqnarray}
\end{subequations}
\begin{Remark}\label{Poincare}
By combaining \eqref{GCRCond1} and \eqref{GCRCond4} with Cartan's first structural equation \eqref{CartansFirstStructural} for $i=1$, we obtained $d\theta_1=0$, i.e., $\theta_1$ is closed. Thus, Poincar\`e lemma implies that $d\theta_1$ is exact, i.e., there exists a function $s$ such that $\theta_1=ds$ which implies $e_1=\partial_s.$
\end{Remark}

\section{Classifications of GCR hypersurfaces in $\mathbb E^4$}

Let $M$ be a GCR hypersurfaces in $\mathbb E^4$. Then, from \eqref{GCRCond4} we have 
\begin{equation}
\label{E4GCRCond0} \omega_{12}(e_3)=\omega_{13}(e_2)=0
\end{equation}
Next, we apply the Codazzi equation \eqref{MinkCodazzi} for $X=e_i,Y=e_j,Z=e_k$ for each triplet (i, j, k) in the set
$\{(1,2,1),$ $(1,3,1),$ $(2,1,2),$ $ (3,1,3), $ $(2,1,3),(3,1,3), $ $(2,3,2)\}$
and combine equations obtained with \eqref{GCRCondall12} and \eqref{GCRCondall34} to get
\begin{subequations}\label{E4GCRCond1}
\begin{eqnarray}
\label{E4GCRCond1a} e_2(k_1)&=&e_3(k_1)=0,\\
\label{E4GCRCond1b} e_1(k_2)&=&\frac{1+\mu\cos\theta k_2}{\mu\sin\theta}(k_1-k_2),\\
\label{E4GCRCond1c} e_1(k_3)&=&\frac{1+\mu\cos\theta k_3}{\mu\sin\theta}(k_1-k_3),\\
\label{E4GCRCond1d} \omega_{23}(e_1) (k_2-k_3)&=&0,\\
\label{E4GCRCond1e} e_2(k_3)&=&\omega_{23}(e_3) (k_2-k_3),\\
\label{E4GCRCond1f} e_3(k_2)&=&\omega_{23}(e_2) (k_2-k_3).
\end{eqnarray}
On the other hand, because of \eqref{GCRCond1}, we have $\langle\nabla_{e_1}e_2,e_1\rangle=0$ which implies $\langle[e_1,e_2],e_1\rangle=0$. Therefore, we obtain  $[e_1,e_2](k_1)=0$.  By a similar way, one can get $[e_1,e_3](k_1)=0$. By combining these equations with \eqref{E4GCRCond1a}, we obtain
\begin{equation}\label{E4GCRCond1g}
 e_2e_1(k_1)=e_3e_1(k_1)=0.
\end{equation}
\end{subequations}


\subsection{Examples of GCR Hypersurfaces in $\mathbb E^4$}
In this subsection, we obtain some examples of  GCR hypersurfaces in $\mathbb E^4$. We will use these results later.

\begin{Prop}\label{GCRE4Example1}
Let $M$ be a hypercylinder over a rotational surface $\mathbb E^4$ given by  
\begin{equation}\label{E4HypoverRot}
x(s,t,u)=(f(s)\cos t,f(s)\sin t, g(s),u)
\end{equation}
Then, $M$ is a GCR hypersurface if and only if it is congruent to one of the following hypersurfaces.
\begin{enumerate}
\item[(i)] A hyperplane $\mathbb E^3$,

\item[(ii)] A spherical hypercylinder $\mathbb S^{2}(r^2)\times \mathbb E$, 

\item[(iii)] A circular hypercylinder $\mathbb S^{1}(r)\times \mathbb E^2$, 

\item[(iv)] A hypercylinder over a cone given by
\begin{equation}\label{CncHyperclinder}
x(s,t,u)=((c_1s+c_2)\cos t,(c_1s+c_2)\sin t, c_2s,u) 
\end{equation}
for some non-zero constants $c_1$ and $c_2$.
\end{enumerate}
All of these hypersurfaces have two distinct principal curvatures.
\end{Prop}
\begin{proof}
Without loss of generality, we may assume $f'^2+g'^2=1.$ By a simple computation, one can obtain that the unit normal vector field of $M$ is $N=(-g'(s)\cos t,-g'(s)\sin t, f(s),0)$ while its principal directions are $\partial_s,\partial_t,\partial_u$ with corresponding principal curvatures of $M$ are $\kappa,-g'/g,0$, where $\kappa$ is the curvature of the curve $\alpha=(f,g)$ on $\mathbb E^2$. By decomposing $x$, we have
$$x^T=(ff'+gg')\partial_s+s\partial_u.$$
Therefore, $M$ is GCR if and only if either $ff'+gg'=0$ or $\kappa=0$. 

If $ff'+gg'=0$ shows then, $\alpha$ becomes a circle which gives the case (ii). If $\kappa=0$, then $\alpha$ is a line. Thus, we have \eqref{CncHyperclinder} for some constants $c_1$ and $c_2$ such that $c_1^2+c_2^2=1$. Note that $c_1=0$ and $c_2=0$ give cases (iii) and (i) of the theorem, respectively. The remaining case shows that $M$ is a conical hypercylinder.

This proves the necessary part of the proposition and the sufficient part follows from a direct computation. 
\end{proof}

We would like to give a generalization of the above result by obtaining the following proposition.
\begin{Prop}\label{GCRE4Example1a}
Let $M$ be a hypersurface $\tilde M\times\mathbb R $ in $\mathbb E^4$, where $\tilde M$ is a surface in $\mathbb E^3$. Then, $M$ is a $GCR$ hypersurface if and only if $\tilde M$ is a sphere, a plane, a cylinder or a tangent developable surface.
\end{Prop}

\begin{proof}
Since a spherical hypercylinder is GCR, we ignore this case and, without loss of generality, we assume that the position vector of $M$ is $x(s,t,u)=(\tilde x(s,t),u)$, where $\tilde x(s,t)=(x_1(s,t),x_2(s,t),x_3(s,t))$ is the position vector of $\tilde M$. By a similar method to proof of Proposition \ref{GCRE4Example1}, we deduce that $\tilde x$ must be a flat GCR surface in $\mathbb E^3$. By using \cite[Proposition 3.3]{YuFu2014GCRS}, we obtain remaining cases.
\end{proof}

\begin{Example}\label{ExO2O2InvHypSurf}
Let $M$ be an $SO(2)\times SO(2)$-invariant hypersurface in $\mathbb E^4$ given by
\begin{equation}\label{O2O2InvHypSurf}
x(s,t,u)=(f(s)\cos t,f(s)\sin t, g(s)\cos u,g(s)\sin u).
\end{equation}
Then, the unit normal vector field of $M$ is $N=(-g'(s)\cos t,-g'(s)\sin t, f'(s)\cos u,f'(s)\sin u)$ and principal curvatures of $M$ are $\kappa,\ -f'/g$ and $g'/f$ with  corresponding principal directions $\partial_s,\partial_t,\partial_u$, where $\kappa$ is the curvature of the profile curve  $\alpha=(f,g)$ in $\mathbb E^2$.
\end{Example}

\begin{Remark}
See \cite{MontaldoAnnali} for several geometrical properties of $SO(p)\times SO(q)$-invariant hypersurfaces in  Euclidean spaces.
\end{Remark}

We will use the following lemma.
\begin{Prop}\label{GCRE4Example2}
Let $M$ be the hypersurface given by \eqref{O2O2InvHypSurf} for some smooth functions $f,g$. Then, $M$ is a GCR hypersurface.
\end{Prop}
\begin{proof}
Without loss of generality, we may assume $f'^2+g'^2=1$. In this case, we have $x(s,t,u)=(ff'+gg')\partial_s+(-fg'+f'g)N$.  Moreover, $e_1=\partial_s$ is a principal direction. Hence,  $M$ is a GCR hypersurface.
\end{proof}

Similarly, we have
 \begin{Prop}\label{GCRE4Example3}
Let $M$ be a rotational hypersurface given by
\begin{equation}\label{RotHypSurf}
x(s,t,u)=(f(s),g(s)\cos t, g(s)\sin t\sin u, g(s)\sin t\cos u)
\end{equation}
for some smooth functions $f,g$ in $\mathbb E^4$. Then, $M$ is a GCR hypersurface.
\end{Prop}


Before we proceed to our main results, we would like to give two examples of GRC hypersurfaces with vanishing Gauss-Kronecker curvature.

\begin{Example}\label{ProgTangHypsurf}
Let $y:\Omega\rightarrow\mathbb E^4$ be an oriented regular surface in $\mathbb S^3(1)\subset\mathbb E^4$ with the spherical normal $n$, where $\Omega$ is an open subset in $\mathbb R^2$. Consider the hypersurface $M$ given by
\begin{equation}\label{ProgTangHypsurfPosVect}
x(s,v,w)=sy(v,w)+cn(v,w),
\end{equation}
where $c$ is a constant. Then, by a direct computation, one can check that $\partial_s$ is a principal direction with the corresponding principal curvature $k_1=0$. Moreover, the unit normal vector field of $M$ is $N=n.$ Since the position vector of this hypersurface is expressed as $x=s\partial_s+cN$, it is a GCR hypersurface. Moreover, its Gauss-Kronecker curvature vanishes identically.
\end{Example}

\begin{Example}\label{ProgTangHypsurfType2}
Let $\alpha(w)$ be a unit speed curve lying on  $\mathbb S^3(1)\subset\mathbb E^4$ and $A(w), B(w)$ two orthonormal vector fields spanning the normal space of $\alpha$ in  $\mathbb S^3(1)$, i.e.,
\begin{equation}\label{ProgTangHypsurfType2Eq1}
\langle A,B\rangle=\langle A,\alpha\rangle=\langle A,\alpha'\rangle=\langle B,\alpha\rangle=\langle B,\alpha'\rangle=0,\quad \langle A,A\rangle=\langle B,B\rangle=1.
\end{equation}
Consider the hypersurface in $\mathbb E^4$ given by
\begin{equation}\label{ProgTangHypsurfType2PosVect}
x(s,v,w)=s\alpha(w)+c\left(\cos \frac vc A(w)+\sin \frac vc B(w)\right)
\end{equation}
for a constant  $c>0$. Then, by a direct computation, one can check that the unit normal vector field of $M$ is $\displaystyle N= \cos \frac vc A(w)+\sin \frac vc B(w)$ and principal directions of $M$ are $e_1=\partial_s,e_2=\partial_v$  $\displaystyle e_3=\frac{1}{\|x_w\|}\partial_w$ with the corresponding principal curvatures $0,-1/c,k_3$. Hence, $M$ is a GCR hypersurface with $\widetilde\nabla_{e_1}e_1=\widetilde\nabla_{e_2}e_1=0$ and $=\widetilde\nabla_{e_2}e_2=-N/c$

\end{Example}
\subsection{$\delta(2)$-ideal hypersurfaces}
In this subsection, we obtain classifications  of $\delta(2)$-ideal GCR hypersurfaces.

First of all, in \cite{ChenIDeal2E4}, Chen proved that a $\delta(2)$-ideal hypersurface in $\mathbb E^4$ is congruent to either a spherical cylinder or a rotational hypersurface given by \eqref{RotHypSurf} for some particularly chosen functions $(f,g)$. Therefore, by combining  \cite[Theorem 4.1]{ChenIDeal2E4} with Proposition \ref{GCRE4Example1} and Proposition \ref{GCRE4Example3}, we obtain  
\begin{Prop}
Let $M$ be a hypersurface in the Euclidean space  $\mathbb E^4$ with two distinct principal curvatures. If $M$ is a $\delta(2)$-ideal hypersurface, then it is GCR.
\end{Prop}

In the following theorem, we focus on hypersurfaces with three distinct distinct principal curvatures. 
\begin{theorem}\label{POHAHA_SPR}
Let $M$ be  a $\delta(2)$-ideal hypersurface in the Euclidean space  $\mathbb E^4$ with three distinct principal curvatures. Then, $M$ is a GCR hypersurface if and only if it is congruent to a $SO(2)\times SO(2)$-invariant hypersurface given by \eqref{O2O2InvHypSurf}.
\end{theorem}
\begin{proof}
Let $M$ be a $\delta(2)$-ideal GCR hypersurface in  $\mathbb E^4$, $e_1,e_2,e_3$ principal directions with corresponding principal curvatures $k_1,k_2,k_3$ and $\omega_1,\omega_2,\omega_3$  1-forms such that $\omega_i(e_j)=\delta_{ij}$, where $e_1$ is proportional to the tangential component $x^T$ of the position vector $x$ of $M$. 

Since $M$ is a $\delta(2)$-ideal, Lemma \ref{Delta2InvLemma} implies either $k_3=k_1+k_2$ or $k_1=k_2+k_3$. First, we want to show 
\begin{equation}\label{E4GCRDelta3Claim1}
e_2(k_2)=e_3(k_2)=e_2(k_3)=e_3(k_3)=0.
\end{equation}

\textit {Case 1.} $k_3=k_1+k_2$. In this case, since $M$ has three distinct principal curvatures, we may assume that the functions $k_1$, $k_2$ and $k_1-k_2$ does not vanish. Since $k_3=k_1+k_2$, \eqref{E4GCRCond1c}  becomes
\begin{equation}
\label{E4GCRCond1c2} e_1(k_1+k_2)=\frac{1+\mu\cos\theta (k_1+k_2)}{\mu\sin\theta}(-k_2).
\end{equation}
By combining \eqref{E4GCRCond1b} and \eqref{E4GCRCond1c2} we obtain 
\begin{equation}
\label{E4GCRDelta3Cond1} e_1(k_1)=\frac{-k_1-2k_1k_2\mu\cos\theta }{\mu\sin\theta}
\end{equation}
Since $k_1\neq0$, applying $e_2$ and $e_3$ to \eqref{E4GCRDelta3Cond1} and using  \eqref{GCRCond3}, \eqref{E4GCRCond1a}, \eqref{E4GCRCond1g} we get \eqref{E4GCRDelta3Claim1}.

\textit {Case 2.} $k_1=k_2+k_3$. In this case, from \eqref{E4GCRCond1b} and \eqref{E4GCRCond1c} we have
\begin{equation}\label{E4GCRDelta3Case2Eq1}
e_1(k_1)=e_1(k_2+k_3)=\frac{k_1+2\mu\cos\theta k_2k_3}{\mu\sin\theta}.
\end{equation}
Similar to the previos case, by applying $e_i$  to \eqref{E4GCRDelta3Case2Eq1}, we obtain 
\begin{equation}\label{E4GCRDelta3Case2Eq2}
e_i(k_2k_3)=0,\quad i=2,3.
\end{equation}
In addition, because of \eqref{E4GCRCond1a} and $k_1=k_2+k_3$, we have 
\begin{equation}\label{E4GCRDelta3Case2Eq3}
e_i(k_2+k_3)=0,\quad i=2,3.
\end{equation}
Therefore, since $M$ has three distinct principal curvatures, \eqref{E4GCRDelta3Case2Eq2} and \eqref{E4GCRDelta3Case2Eq3} imply \eqref{E4GCRDelta3Claim1}.

Hence, we have proved \eqref{E4GCRDelta3Claim1} in both cases. Therefore, \eqref{GCRCond2} and \eqref{E4GCRCond1d}-\eqref{E4GCRCond1f} imply
\begin{subequations}\label{E4GCRDelta3Claim2All}
\begin{eqnarray}
\label{E4GCRDelta3Claim2a}\omega_{23}=0,&&\\
\label{E4GCRDelta3Claim2b}\omega_{12}=\omega_{12}(e_2)\omega_2,&\quad& \omega_{13}=\omega_{13}(e_3)\omega_3.
\end{eqnarray}
In addition, from \eqref{GCRCond4} we have
\begin{equation}\label{E4GCRDelta3Claim2c}
e_i(\omega_{1j}(e_j))=0,\quad i=2,3.
\end{equation}
\end{subequations}

Next, we put $X=e_1$ and define two distributions $T_1=\mathrm{span}\{e_2\}$ and $T_2=\mathrm{span}\{e_3\}$. By using \eqref{E4GCRDelta3Claim2All}, one can check that $X,T_1,T_2$  satisfy hypothesis of Lemma \ref{WarpPrdctKeyLemma} from which we see that there exists a local coordinate system $(s,t,u)$ such that $e_2$ and  $e_3 $ are proportional to $\partial_t$ and $\partial_u$, respectively. Furthermore, the metric tensor is $g=ds^2+a(s)A(t)dt^2+b(s)B(u)du^2$ for some smooth functions $a,b,A,B$  and 
\begin{equation}\label{E4GCRDelta3FindingxEq0}
\nabla_{\partial_u}\partial_t=\nabla_{\partial_t}\partial_u=0.
\end{equation}
By re-defining $t$ and $u$, we may assume $A(t)=B(u)=1$.  Hence, $g$ becomes
\begin{equation}\label{E4GCRDelta3FindingxEq2}
g=ds^2+a(s)dt^2+b(s)du^2.
\end{equation}

On the other hand, \eqref{E4GCRDelta3FindingxEq0}  implies $\widetilde\nabla_{\partial_u}\partial_t=\widetilde\nabla_{\partial_t}\partial_u=0$. Therefore, the position vector $x$ has the form
\begin{equation}\label{E4GCRDelta3FindingxEq1}
x(s,t,u)= x_1(s,t)+x_2(s,u).
\end{equation}
By a further computation using \eqref{E4GCRDelta3FindingxEq2}, we obtain 
$\widetilde\nabla_{\partial_s}\partial_t=a'/a\partial_t.$
By combining this equation with \eqref{E4GCRDelta3FindingxEq1} we obtain $\partial_s\partial_t(x_1)=\partial_t(x_1)$ which implies $x_1(s,t)=f(s)\theta_1(t)+ \alpha_1(s)$ for a smooth function $f$ and some vector values function $\theta_1,\alpha_1$. By a similar way, we get also $x_2(s,u)=g(s)\theta_2(u)+ \alpha_2(s)$. Thus, \eqref{E4GCRDelta3FindingxEq1} implies 
\begin{equation}\label{E4GCRDelta3FindingxEq3}
x(s,t,u)= \Gamma(s)+f(s)\theta_1(t)+g(s)\theta_2(u),
\end{equation}
where $\Gamma=\alpha_1+\alpha_2$. Now, let $p=x(s_0,t_0,u_0)\in M$. Note that the slice $s=s_0, u=u_0$ (resp. $s=s_0, t=t_0$) is the integral curve of $\partial_t$ (resp. $\partial_u$) passing through $p$. Since $\nabla_{e_i}e_i=0,\ i=2,3$ because of \eqref{E4GCRDelta3Claim2a}, \eqref{E4GCRDelta3FindingxEq2} gives $\widetilde\nabla_{\partial_t}\partial_t=k_2(s)\partial_t$ and $\widetilde\nabla_{\partial_u}\partial_u=k_3(s)\partial_u$. Recall that $k_2$ and $k_3$ are constant along any integral curve of $\partial_t$ and $\partial_u$. Therefore, they must be circles or lines subject to $k_i=0$ or $k_i\neq0$.Thus, we have two cases because $M$ has three distinct principal curvatures.

\textit {Case A.} $k_2,k_3\neq 0$. In this case, integral curves of $\partial_t$ and $\partial_u$ are circles. Thus, a further computation yields that $M$ is congruent to \eqref{O2O2InvHypSurf}.

\textit {Case B.} $k_2\neq 0$, $k_3=0$. In this case, the integral curves of $\partial_t$ are circles while the integral curves of $\partial_u$ are lines. Thus, by a further computation one can obtain that $M$ is congruent to \eqref{E4HypoverRot}. However, Proposition \ref{GCRE4Example1} implies that this case is not possible because $M$ has three distinct principal curvatures.

Hence, the proof is completed.
\end{proof}
\begin{Remark}
See also  \cite{HsanisField,TurgayHHypersurface} for a classification of hypersurfaces satisfying \eqref{E4GCRDelta3Claim2All} under the restriction of being biconservative.
\end{Remark}

For the existence of $\delta(2)$-ideal GCR hypersurfaces, we state the following corollary of Theorem \ref{POHAHA_SPR}.
\begin{Corol}
Let $M$ be a hypersurface in the Euclidean space  $\mathbb E^4$ with three distinct curvature. Then, $M$ is a 1-minimal, $\delta(2)$-ideal GCR hypersurface if and only if it is congruent to either an open part of a hyperplane or a hypersurface given by
\begin{equation}\label{LBLHyp3min1min}
x(s,t,u)=({\sqrt2}s\cos t,{\sqrt2}s\sin t, {\sqrt2}s\cos u,{\sqrt2}s\sin u).
\end{equation}
\end{Corol}
\begin{Remark}
The hypersurface given by \eqref{LBLHyp3min1min} has the shape operator $S=\diag(0,k_2,-k_2)$. Thus, it is a 1-minimal and 3-minimal hypersurface with non-constant second mean curvature. Furthermore, this  $SO(2)\times SO(2)$ invariant hypersurface is also a member of family of hypersurfaces given in Example \ref{ProgTangHypsurf}, because it can be obtain by putting $y(v,w)=({\sqrt2}\cos v,{\sqrt2}\sin v,$ $ {\sqrt2}\cos w,{\sqrt2}\sin w)$ and $C=0$ in \eqref{ProgTangHypsurfPosVect}.
\end{Remark}
\subsection{Hypersurfaces with constant mean curvature}
In this subsection, we consider GCR hypersurfaces with a constant mean curvature.

First, we focus on hypersurface whose first mean curvature is constant and obtain the following classification.
\begin{theorem}
Let $M$ be a hypersurface in $\mathbb E^4$ with  constant first mean curvature. Then, $M$ is a GCR hypersurface if and only if it is congruent to one of the following 3 types of hypersurfaces.
\begin{enumerate}
\item [(i)] An isoparametric hypersurface,
\item [(ii)] A rotational hypersurface given by \eqref{RotHypSurf},
\item [(iii)] An  $SO(2)\times SO(2)$-invariant hypersurface  given by \eqref{O2O2InvHypSurf}.
\end{enumerate}

\end{theorem}
\begin{proof}
Let $M$ be a GCR hypersurface in $\mathbb E^4$ with principal curvatures $k_1,k_2,k_3$ with corresponding principal directions $e_1,e_2,e_3$, where $ e_1=x^T/\|x^T\|.$ Also assume that $M$ has constant first mean curvature, i.e.,
\begin{equation}\label{E4H1ConstEq1}
k_1+k_2+k_3=c
\end{equation}
for a constant $c$. We have three cases subject to number of distinct principal curvatures of $M$.

\textit{Case I.} $M$ has only one distinct principal curvature. In this case, $M$ is either a hypersphere or hyperplane which gives the case (i) of the theorem.

\textit{Case II. } $M$ has two distinct principal curvatures. In this case, we have either $k_1=k_2$ or $k_2=k_3$.

\textit{Case IIa. }$k_1=k_2$. In this case, an instance of Codazzi equation \eqref{MinkCodazzi} yields $e_1(k_1)=0$. However, this equation and \eqref{E4GCRCond1a} implies that $k_1$ is a constant. However, \eqref{E4H1ConstEq1} implies $k_3$ is also constant. Therefore, $M$ is an isoparametric hypersurface. Thus,  we have the case (ii) of the theorem.

\textit{Case IIb.} $k_2=k_3$. In this case, an instance of Codazzi equation \eqref{MinkCodazzi} yields $e_2(k_2)=e_2(k_3)=0$. Thus, we have $k_i=k_i(s),\ i=1,2,3$, where $(s,\hat t,\hat u)$ is  a local coordinate system such that $s$ is choosen as described in Remark \ref{Poincare}. Hence, by the inverse function theorem, we, locally, have $k_2=k_3$ and $k_1=\lambda(k_2)$. Therefore, the well-known result of do Carmo and Dajczer implies that $M$ is a rotational hypersurface (see \cite[Theorem 4.2]{doCarmoRotHyper}). Hence, we have the case (ii) of the theorem.

\textit{Case III.} $M$ has three distinct principal curvatures. In this case, by combining \eqref{E4GCRCond1c} with \eqref{E4H1ConstEq1}, we obtain
\begin{equation}\nonumber
e_1(-k_1-k_2+c)=\frac{1+\mu\cos\theta k_3}{\mu\sin\theta}(2k_1+k_2-c).
\end{equation}
By combining this equation and \eqref{E4GCRCond1c}, we get
$$\frac{-\mu\sin\theta e_1(k_1)-3k_1+c}{\mu\cos\theta}+c^2-3ck_1+2k_1^2=-2k_1k_2-2k_2^2+2ck_2.$$
By applying $e_i$ to this equation and combining the equation obtained with \eqref{GCRCond3}, \eqref{E4GCRCond1a} and \eqref{E4GCRCond1g}, we obtain
\begin{equation}\nonumber
e_i(k_2)(2k_2+k_1-c)=0.
\end{equation}
for $i=2,3$. Thus, on the open subset $\mathcal U_i=\{p|e_i(k_2)(p)\neq0\}$ we have $2k_2+k_1-c=0$. By applying $e_i$ to this equation and considering \eqref{E4GCRCond1a}, we obtain $e_i(k_2)=0$ on  $\mathcal U_i$ which yields a contradiction unless $\mathcal U_i=\varnothing$. Hence, we have $e_i(k_2)=0$ on $M$. Moreover,  \eqref{E4GCRCond1a} and \eqref{E4H1ConstEq1} imply $e_i(k_3)=0$. 

We have proved that corresponding connection forms of $M$ satisfy \eqref{E4GCRDelta3Claim2All}. Hence, $M$ is an open part of $SO(2)\times SO(2)$-invariant hypersurface (See the proof of Theorem \ref{POHAHA_SPR}) which gives the case (iii) of the theorem.

This proves the necessary part of the proposition and the sufficient part follows from a direct computation. 
\end{proof}

Next, we consider 3-minimal hypersurfaces, i.e., hypersurfaces with the Gauss-Kronecker curvature vanishing identically. We obtain the following classification theorem.

\begin{theorem}
Let $M$ be a 3-minimal hypersurface in the Euclidean space $E^4$. Then, $M$ is a GCR hypersurface if and only if it is congruent to one of the following 3 types of hypersurfaces.
\begin{enumerate}
\item [(i)] One of hypercylinders given in Proposition \ref{GCRE4Example1},

\item [(ii)] A hypersurface given in Example \ref{ProgTangHypsurf} for a regular surface $y$ lying on $\mathbb S^3(1)$,

\item [(iii)] A hypersurface given in Example \ref{ProgTangHypsurfType2} for a regular curve $\alpha$ lying on $\mathbb S^3(1)$.
\end{enumerate}
\end{theorem}
\begin{proof}
Let $M$ be a  3-minimal GCR hypersurface and $e_1=x^T/\|x^T\|,e_2,e_3$ its principal directions with corresponding principal curvatures $k_1,k_2,k_3$. We consider  a local, orthogonal coordinate system $(s,v,w)$, where $s$ is the coordinate function described in Remark \ref{Poincare}, i.e., $e_1=\partial_s$.

First, we want to show that  $k_1$ vanishes identically on $M$. Assume, towards contradiction, that $k_1(m)\neq 0$ at a point $m\in M$. In this case, there exists a neighborhood $\mathcal N_m$ of $m$ on which $k_1$ does not vanish. Then, since the Gauss-Kronecker curvature $k_1k_2k_3=0$, by shrinking if necessary, we may assume $k_2=0$  on $\mathcal N_m$. Note that we also have  $e_1(k_2)=0$ on $\mathcal N_m$. However, if we use these two equations in \eqref{E4GCRCond1b}, we  obtain $k_1=0$ on $\mathcal N_m$ which is a contadiction. Thus, we have $k_1=0$ on $M$.

Since $k_1=0$, from  \eqref{GCRCond1} we obtain 
\begin{equation}\label{GaussKronCurvEq1}
\widetilde\nabla_{\partial_s}\partial_s=0,\quad \widetilde\nabla_{\partial_s}N=0
\end{equation}
and \eqref{GCRCond2} implies $e_1(\theta)=\cos\theta/\mu$. Thus, from \eqref{DexompofxRes2} for $X=e_1$  we have
\begin{equation}\label{GaussKronCurvEq1b}
\mu\cos\theta=c
\end{equation}
for a constant $c$. On the other hand, by applying $e_1=\partial_s$ to the equation $\langle \partial_s,x\rangle=\mu\sin\theta$ and using  \eqref{GaussKronCurvEq1}, we obtain $e_1(\mu\sin\theta)=1$. Therefore, by translating $s$ if necessary, we may assume
\begin{equation}\label{GaussKronCurvEq1c}
\mu\sin\theta=s.
\end{equation}
Hence, by combining \eqref{GaussKronCurvEq1}, \eqref{GaussKronCurvEq1b}  and \eqref{GaussKronCurvEq1c} with \eqref{Dexompofx}, we get  
\begin{equation}\label{GaussKronCurvEq2}
x(s,v,w)=s e_1(v,w) +c  N(v,w),
\end{equation}
where,  by abuse of notation,  we denote the Euclidean coordinates of the vector field $e_1$ by $e_1(v,w)$. 

Now, we want to consider  three cases separately.
 
\textit {Case I.} $\nabla_{e_2} e_1\neq0,\ \nabla_{e_3} e_1\neq0$. In this case, we put $y(v,w)=e_1(v,w)$. Then, because of assumptions, $y$ defines a regular surface in $\mathbb S^3(1)\subset\mathbb E^4$. Moreover, since $N$ is the unit normal vector field of $M$, we have $\langle\partial_s, N\rangle=\langle\partial_v, N\rangle=\langle\partial_w, N\rangle=0$. Thus, from \eqref{GaussKronCurvEq2} we get 
$$\langle y, N\rangle=\langle y_v, N\rangle=\langle y_w, N\rangle=0.$$
Therefore, $n(v,w)=N(v,w)$ is the spherical unit normal vector field of the regular surface $y$. Thus, \eqref{GaussKronCurvEq2} turns into \eqref{ProgTangHypsurfPosVect} given in the Example \ref{ProgTangHypsurf}. Hence, we have the case (ii) of the theorem.

\textit {Case II.} $\nabla_{e_2} e_1=\nabla_{e_3} e_1=0$. In this case, \eqref{GCRCond4} for $i=2,3$ and \eqref{GaussKronCurvEq1b} imply $c\neq0$ and $k_2=k_3=-1/c$. Therefore, principal curvatures of $M$ are obtained as $0,-1/c,-1/c$. Hence, $M$ is a spherical hypercylinder. Thus, we have the case (i) of the theorem.

\textit {Case III.} $\nabla_{e_2} e_1=0,\ \nabla_{e_3} e_1\neq0$. In this case, we have $\omega_{12}(e_2)=0$ which yields $\omega_{12}=0$ in view of \eqref{E4GCRCond0}. Thus, \eqref{GCRCond4} for $i=2$ and \eqref{GaussKronCurvEq1b} imply $c\neq0$ and  $k_2=-1/c$. Thus, we have $e_3(k_2)=0$ and, obviously, $k_2\neq k_3$. Therefore, by using \eqref{E4GCRCond1f}, we obtain $\omega_{23}(e_2)=0$. Furthermore, since  $k_2\neq k_3$, \eqref{E4GCRCond1d} implies $\omega_{23}(e_1)=0$. Since $\omega_{12}=0$ and $\omega_{23}=\omega_{23}(e_3)\theta_{3}$, by using \eqref{CartansFirstStructural} for $i=2$, we obtain $d\theta_2=0$. Therefore, by applying the Poincar\`e lemma,  we see that we may assume $e_2=\partial_v$. Moreover, since $(s,v,w)$ is orthogonal, we also have $e_3$ is proportional to $\partial_w$.

On the other hand, since $\widetilde\nabla_{e_2} e_1=0$, \eqref{GaussKronCurvEq2} becomes
\begin{equation}\label{GaussKronCurvCaseIIIEq1}
x(s,v,w)=s e_1(w) +c  N(v,w).
\end{equation}
In addition, since  $\omega_{23}(e_1)=\omega_{23}(e_2)=0$ and $k_2=-1/c$, we have 
\begin{equation}\label{GaussKronCurvCaseIIIEq2}
\widetilde\nabla_{\partial_v}\partial_v=\widetilde\nabla_{e_2}e_2=-\frac{1}cN.
\end{equation}
From \eqref{GaussKronCurvCaseIIIEq1} and \eqref{GaussKronCurvCaseIIIEq2} we get 
\begin{equation}\nonumber
N_{vv}+\frac{1}{c^2}N=0
\end{equation}
which implies 
\begin{equation}\label{GaussKronCurvCaseIIIEq3}
N=\cos \frac vc A(w)+\sin \frac vc B(w)
\end{equation}
for some smooth vector valued functions $A,B$. Now, we put $\alpha(w)=e_1(w)$ in \eqref{GaussKronCurvCaseIIIEq1} and use \eqref{GaussKronCurvCaseIIIEq3} to get \eqref{ProgTangHypsurfType2PosVect}. Note that if $\alpha(w)$ is constant $M$ becomes a hypercylinder which gives the case (i) of the theorem. 

Next, we assume that $\alpha'(w)$ does not vanish. Therefore, $\alpha$ is a regular curve satisfying $\langle \alpha,\alpha \rangle=\langle e_1,e_1\rangle=1$. Thus, by re-defining $w$ properly, we may assume $\alpha$ is a unit speed curve lying on $\mathbb S^3(1)$ and we can consider $A$, $B$ as vector fields on $\alpha$. Moreover, since $(s,v,w)$ is orthogonal, we have $\langle x_s,x_v\rangle=\langle x_v,x_w\rangle=\langle x_s,x_w\rangle=0$. By combining these equations with \eqref{ProgTangHypsurfType2PosVect}, we get  \eqref{ProgTangHypsurfType2Eq1}. Hence, we obtain the case (iii) of the theorem. 

Converse is obivous. Hence, the proof is completed.
\end{proof}

\end{document}